\documentclass[12pt]{article}
\usepackage{geometry} 
\usepackage{amsmath, amssymb, amsthm, graphicx}
\usepackage{color}
 
\usepackage{cases}

\usepackage{enumitem}
\DeclareMathOperator{\esssup}{ess \, \, sup}
\DeclareMathOperator{\essinf}{ess \, \, inf}

\def \e {\varepsilon}

\newcommand{\R}{{\mathbb R}}

\newtheorem{theorem}{Theorem}[section]
\newtheorem{lemma}[theorem]{Lemma}
\newtheorem{definition}[theorem]{Definition}

\newtheorem{cor}[theorem]{Corollary}
\newtheorem{prop}[theorem]{Proposition}

\theoremstyle{remark}
\newtheorem{rem}[theorem]{Remark}

\usepackage{enumitem}

\title{Large time behavior and Lyapunov functionals for a nonlocal differential equation}
\author{Danielle Hilhorst   \thanks{\texttt{CNRS, Laboratoire de Math\'{e}matique d'Orsay, Analyse Num\'{e}rique et EDP,
	 Universit\'e 	de Paris-Sud, Universit\'e Paris-Saclay, F-91405 Orsay Cedex, France}}, \
	 Philippe Lauren\c cot \thanks{\texttt{Institut de Math\'ematiques de Toulouse, UMR 5219, Universit\'e 
	 de Toulouse, CNRS, F-31062 Toulouse Cedex 9, France}}, \ 
	 Thanh-Nam Nguyen \thanks{\texttt{National Institute for Mathematical Sciences,
	 		70, Yuseong-daero 1689 beon-gil, Yuseong-gu, Daejeon, 34047, Korea}}
	}
\date{\today} 

\begin{document}

\maketitle

\begin{abstract}
A new approach is used to describe the large time behavior of the nonlocal differential equation initially studied in \cite{thanhnam}. Our approach is based upon the existence of infinitely many Lyapunov functionals and allows us to extend the analysis performed in \cite{thanhnam}.
	
	\medskip 

	
\end{abstract}

\bigskip

\section{Introduction}

Let $\Omega$ be an open bounded subset of  $\R^N$,  $N \ge 1$, and consider the following nonlocal differential equation
\begin{equation*}\label{chap3:3chapchaporiginal:eq}
(P) \ \ \ \left\{ 
\begin{aligned}
&u_t =  u^2(1-u)-u(1-u) \,\frac{\displaystyle\int_\Omega u^2(1-u)\ dx}{\displaystyle\int_\Omega u(1-u)\ dx}\ ,   &    & \quad \quad x\in \Omega,\  t \ge 0 \ ,
\\
&u(x,0)=u_0(x)\ , & & \quad \quad x\in \Omega\ ,
\end{aligned}
\right.
\end{equation*}
which was proposed by M.~Nagayama \cite{nagayama} to describe bubble motion with chemical reaction when a volume constraint is included. The initial condition $u_0$ is here a bounded function in $L^\infty(\Omega)$. Note that Problem $(P)$ is of bistable type since it can be written in the form
$$
u_t=u(1-u)(u-\lambda(t)),
$$
where \begin{equation}
\lambda(t):=\frac{\displaystyle\int_\Omega u^2(1-u)(x,t)\ dx}{\displaystyle\int_\Omega u(1-u)(x,t)\ dx} \ . \label{lambda}
\end{equation}

\medskip

A general form of Problem $(P)$ is actually studied in \cite{thanhnam} where the well-posedness and the large time behavior of solutions of $(P)$ are investigated. In particular, the structure of the $\omega$-limit sets of solutions of $(P)$ is described in \cite{thanhnam} with the help of the rearrangement theory. Restricting our attention to the specific Problem $(P)$ given above, the aim of this paper is to provide additional information on the $\omega$-limit sets with an alternative approach. Our approach is actually based upon the existence of infinitely many Lyapunov functionals, from which we deduce the limit of the nonlocal term and hence the $\omega$-limit set.

As in \cite{thanhnam}, we suppose that the initial condition $u_0$ satisfies one of the following hypotheses:

\begin{itemize}
\item[$\bf (H_1)$] \,
$u_0 \in L^\infty(\Omega)$, $u_0(x) \ge 1 \mbox{~~for a.e.~~}x\in\Omega,  \mbox{~~and~~} u_0 \not\equiv 1.$

\item[$\bf (H_2)$] \, 
$u_0 \in L^\infty(\Omega)$, $0 \le  u_0(x) \le 1 \mbox{~~for a.e.~~}x\in\Omega,  \mbox{~~and~~} u_0(1-u_0) \not\equiv 0.$

\item[$\bf (H_3)$] \, 
$u_0 \in L^\infty(\Omega)$, $u_0(x) \le 0 \mbox{~~for a.e.~~}x\in\Omega,  \mbox{~~and~~} u_0 \not\equiv 0.$
\end{itemize}

In \cite{thanhnam}, when $u_0$ satisfies either $\bf (H_1)$ or $\bf (H_3)$, the solution $u$ of $(P)$ is shown to converge to a step function. The first contribution of this paper is to identify this function in terms of the initial condition $u_0$. We obtain a less precise result when $u_0$ satisfies $\bf (H_2)$ but complete the analysis performed in \cite{thanhnam} in that case. Besides these qualitative results we also identify an infinite family of Lyapunov functionals for $(P)$.

\bigskip

For further use, we define  
$$
f(z) :=z^2(1-z)\ , \qquad  g(z)=z(1-z)\ , \qquad z\in\mathbb{R}\ .
$$ 
Throughout the paper we denote the Lebesgue measure of a measurable set $A \subset \Omega$ by $|A|$. 

The paper is organized as follows. In Section~\ref{sec2}, we recall some results from \cite{thanhnam} dealing with the well-posedness of problem $(P)$ as well as with the existence of invariant sets. In Section~\ref{sec3}, we show that problem $(P)$ possesses infinitely many Lyapunov functionals, and use this property to study the limit of the nonlocal term.
Finally, in Section~\ref{sec4}, we characterize the $\omega$-limit set for initial data satisfying either $\bf (H_1)$ or $\bf (H_3)$, and improve the outcome of \cite{thanhnam} when $\bf (H_2)$ holds.

\section{Well-posedness and $\omega$-limit sets}\label{sec2}

We recall some results from \cite{thanhnam} and first make precise the notion of solution to $(P)$ to be used later.

\begin{definition}[\cite{thanhnam}]\label{thedefinitionofsolution}
Let $0<T \le \infty$. A function $u\in C^1([0,T);L^\infty(\Omega))$ is called a solution of Problem $(P)$ on $[0,T)$ if  the following three properties hold:
\begin{enumerate}[label=\emph{(\roman*)}]
\item $u(0)=u_0$,
\item $\displaystyle \int_\Omega g(u(x,t))\ dx \neq 0 \mbox{~~for all~~} t\in[0, T)$,
\item $u_t(x,t) = f(u(x,t)) - \lambda(t) g(u(x,t))$ for a.e. $x\in\Omega$ and all $t\in[0, T)$,
where $\lambda(t)$ is defined in \eqref{lambda}.

\end{enumerate}
\end{definition}

We note that solutions of $(P)$ on $[0,T)$ satisfy the mass conservation property:
\begin{equation}
\int_\Omega u(x,t)\,dx=\int_\Omega u_0(x)\,dx\ , \quad t \in [0,T)\ . \label{massconservation}
\end{equation} 

We summarize the well-posedness of $(P)$ in the next result.

\begin{prop}[\cite{thanhnam}] \label{thm:existence:boundedness}
Assume that $\bf (H_i)$ holds for some $i=1,2,3$. Problem $(P)$ possesses a global solution $u \in C^1([0, \infty); L^\infty(\Omega))$.
Moreover:
\begin{enumerate}[label=\emph{(\roman*)}]
\item If  $\bf (H_1)$ holds, then for all $t \ge 0$,
\begin{equation}\label{ine:h1}
1 \le u(x,t) \le \esssup_{\Omega}u_0 \mbox{~~a.e. in~~} \Omega.
\end{equation}
\item If $\bf (H_2)$ holds, then for all $t \ge 0$,
\begin{equation}\label{ine:h2}
0 \le u(x, t) \le 1\mbox{~~a.e. in~~}\Omega.
\end{equation}
\item If  $\bf (H_3)$ holds, then for all $t \ge 0$,
\begin{equation}\label{ine:h3}
\essinf_{\Omega} u_0 \le u(x,t) \le 0 \mbox{~~a.e. in~~} \Omega.
\end{equation}
\end{enumerate}
\end{prop}

Given $u_0$ satisfying $\bf (H_i)$ for some $i=1,2,3$, we define the set $I_i$ by 
$$
I_1:=[1, \esssup_{\Omega} u_0], \quad I_2:=[0, 1], \quad I_3:=[\essinf_{\Omega} u_0, 0]\ ,
$$
according to the value of $i$. A consequence of Proposition~\ref{thm:existence:boundedness} is that $I_1$, $I_2$, and $I_3$ are invariant sets for the flow associated to $(P)$. This property entails the boundedness of $\lambda$ as shown below.

\begin{cor}\label{bounds:lambda(t)}
Assume that $u_0$ satisfies $\bf (H_i)$ for some $i=1,2,3$. Then $\lambda(t) \in I_i$ for all $t \ge 0$.
\end{cor}
\begin{proof}
The corollary is a consequence of the uniform bounds for $u$ in Proposition~\ref{thm:existence:boundedness} and the property $f(z)=zg(z)$ for $z\in\mathbb{R}$.
\end{proof}

We finally recall that, given an initial condition $u_0$ satisfying $\bf (H_i)$ for some $i=1,2,3,$ and denoting the corresponding solution to $(P)$ given by Proposition~\ref{thm:existence:boundedness} by $u$, the $\omega$-limit set of $u_0$ is defined in \cite{thanhnam} as follows:
\begin{equation*}
\varphi\in \omega(u_0) \;\text{ if and only if }\; 
\left\{
\begin{array}{l}
\varphi\in L^1(\Omega) \;\text{ and there is a sequence }\; (t_n)_{n\ge 1} \\
\text{ such that }\; \displaystyle{\lim_{n\to\infty} \|u(t_n)-\varphi\|_{L^1(\Omega)} = 0} \text{ and }\\
\displaystyle{\lim_{n\to\infty} t_n = \infty}\ .
\end{array}
\right. 
\end{equation*}

Several properties of $\omega(u_0)$ are obtained in \cite{thanhnam}:

\begin{theorem}[\cite{thanhnam}] \label{thm:oml}
\begin{enumerate}[label=\emph{(\roman*)}]
\item If  $\bf (H_1)$ holds, then $\omega(u_0)=\{\varphi\}$ is a singleton and there are $\mu>1$ and a measurable subset $A$ of $\Omega$ such that
\begin{equation}\label{oml:h1}
\varphi = \mu \chi_{A} + \chi_{\Omega\setminus A}\ .
\end{equation}
\item If $\bf (H_2)$ holds and $\varphi\in \omega(u_0)$, then there are $\nu\in (0,1)$ and two disjoint measurable subsets $A_1$ and $A_2$ of $\Omega$ such that
\begin{equation}\label{oml:h2}
\varphi = \chi_{A_1} + \nu \chi_{A_2}\ .
\end{equation}
\item If $\bf (H_3)$ holds, then $\omega(u_0)=\{\varphi\}$ is a singleton and there are $\xi<0$ and a measurable subset $A$ of $\Omega$ such that
\begin{equation}\label{oml:h3}
\varphi = \xi \chi_{A}\ .
\end{equation}
\end{enumerate}
\end{theorem}

\section{Lyapunov functionals and limit of the nonlocal term}\label{sec3}

\begin{prop}\label{prop:lyapunov:functionals}
Assume that $\bf (H_i)$ holds for some $i =1, 2, 3$. Let $\Phi \in C^1(\R)$ be such that $\Phi'$ is non-decreasing on $I_i$. Then 
$$E_i(u(t)):=(-1)^{i+1}\int_\Omega \Phi(u(x,t))\,dx$$
 is a Lyapunov functional of Problem $(P)$. As  a consequence, $\lim_{t \to \infty} E_i(u(t))$ exists as $t \to \infty$.
\end{prop}

\begin{proof}
We only prove the statement in the case $i=1$. It follows from $(P)$ and \eqref{massconservation} that
\begin{align*}
\frac{d}{dt} E_1(u)&=\int_\Omega \Phi'(u)u_t\,dx=\int_\Omega \Phi'(u)u_t\,dx-\int_\Omega \Phi'(\lambda)u_t\,dx\\
&=\int_\Omega \big(\Phi'(u)-\Phi'(\lambda) \big) u_t\,dx \\
&=\int_\Omega \big(\Phi'(u)-\Phi'(\lambda) \big) \big(u-\lambda \big ) u(1-u)\,dx. 
\end{align*}
Since $u\ge 1$ and $\lambda \ge 1$ by \eqref{ine:h1} and Corollary~\ref{bounds:lambda(t)}, the monotonicity of $\Phi'$ on $I_1$ implies $(\Phi'(u)-\Phi'(\lambda) \big) \big(u-\lambda \big ) \ge 0$. Therefore
$$\frac{d}{dt} E_1(u(t)) \le 0\ ,$$
and $t\mapsto E_1(u(t))$ is a non-increasing function of time. In other words, $E_1$ is a Lyapunov functional for $(P)$. Furthermore, since $u$ is uniformly bounded in time by Proposition~\ref{thm:existence:boundedness}, $t\mapsto E_1(u(t))$ is bounded from below. Hence $\lim_{t \to \infty} E_1(u(t))$ exists.
\end{proof}

\begin{rem}
The above proposition implies that there are infinitely many Lyapunov functionals for Problem $(P)$.
\end{rem}

\begin{cor} \label{cor:lim:nume:denominateur}
Assume that $\bf (H_i)$ holds for some $i =1, 2, 3$. Then
$$
l_g := \lim_{t \to \infty} \int_\Omega g(u(x,t))\ dx \;\;\text{ and }\;\;  l_f := \lim_{t \to \infty} \int_\Omega f(u(x,t))\ dx \;\;\text{ exist.}
$$ 
\end{cor}

\begin{proof}
First we consider the case where $\bf (H_1)$ holds. Recall that in this case $u(x,t) \in I_1$ for a.e. $x \in \Omega$ and all $t \ge 0$. Since the functions $-g'(z)=2z-1$ and $-f'(z) = 3z^2 -2z$ are non-decreasing on $I_1$, we infer from Proposition~\ref{prop:lyapunov:functionals} that
$$
- l_g = \lim_{t \to \infty} \int_\Omega (-g)(u(x,t))\ dx \;\;\text{ and }\;\;  - l_f = \lim_{t \to \infty} \int_\Omega (-f)(u(x,t))\ dx
$$
both exist. Hence the result of the corollary follows in the case that $\bf (H_1)$ holds.

The case where $\bf (H_3)$ holds is proved in a similar way. 

Finally we consider the case of hypothesis $\bf (H_2)$. As in the previous cases, the function $-g'$ is non-decreasing on $I_2$ so that 
$$
l_g  = \lim_{t \to \infty} \int_\Omega g(u(x,t))\ dx \;\;\text{ exists.}
$$
To complete the proof we simply note that $f=f_1-f_2$ with $f_1(z):=z^2$ and $f_2(z):= z^3$. Since $f_1'$ and $f'_2$ are non-decreasing on $I_2$, using again Proposition~\ref{prop:lyapunov:functionals} ensures the existence of the limits 
$$
\lim_{t \to \infty} \int_\Omega f_1(u(x,t))\ dx \;\;\text{ and }\;\; \lim_{t \to \infty} \int_\Omega f_2(u(x,t))\ dx\ ,
$$
from which we readily deduce that
$$
\lim_{t \to \infty} \int_\Omega f(u(x,t))\ dx = \lim_{t \to \infty} \int_\Omega f_1(u(x,t))\ dx - \lim_{t \to \infty} \int_\Omega f_2(u(x,t))\ dx
$$
exists.
\end{proof}

After this preparation we are in a position to study the behavior of $\lambda(t)$ as $t\to\infty$.

\begin{lemma}\label{lem:lambda:infty} \mbox{~~}
\begin{enumerate} [label=\emph{(\roman*)}]
\item Assume that  $\bf (H_1)$ holds. Then $\lambda_\infty:=\lim_{t \to \infty} \lambda(t)$ exists and $\lambda_\infty >1$. 
\item Assume that  $\bf (H_3)$ holds. Then $\lambda_\infty:=\lim_{t \to \infty} \lambda(t)$ exists and $\lambda_\infty <0$. 
\end{enumerate}
\end{lemma}

\begin{proof} 
Owing to Corollary~\ref{cor:lim:nume:denominateur} we set 
$$
l_g:=\displaystyle\lim_{t \to \infty} \int_\Omega g(u(x,t))\ dx \ , \quad l_f:=\displaystyle\lim_{t \to \infty} \int_\Omega f(u(x,t))\ dx \ .
$$
(i) Since $u \ge 1$ by Proposition~\ref{thm:existence:boundedness}, one has $g(u(t)) \le 0$ a.e. in $\Omega$ for all $t \ge 0$, so that $l_g \le 0$. Now we show that $l_g < 0$. Assume for contradiction that $l_g =0$. Then 
$$
-\int_\Omega (u(x,t)-1)\ dx  - \int_\Omega (u(x,t)-1)^2\ dx = \int_\Omega g(u(x,t))\,dx \mathop{\longrightarrow}_{t \to \infty} 0\ . 
$$
Therefore 
$$
u(t) \to 1 \mbox{~~in~~}L^2(\Omega) \mbox{~~as~~}t \to \infty.
$$
We then deduce from the mass conservation property \eqref{massconservation} that
$$
\frac{1}{|\Omega|}\int_\Omega u_0(x)\ dx =1\ ,
$$ 
which contradicts $\bf (H_1)$. Therefore $l_g<0$ hence $\lambda_\infty:=\lim_{t \to \infty} \lambda(t)$ exists. 

It follows from Corollary~\ref{bounds:lambda(t)} that $\lambda_\infty \ge 1$. Next we show that $\lambda_\infty >1$. Assume for contradiction that $\lambda_\infty =1$. Then
\begin{align*}
\lim_{t \to \infty} \int_\Omega u(1-u)^2(x,t)\ dx & = \lim_{t \to \infty}\int_\Omega g(u(x,t))\ dx - \lim_{t \to \infty} \int_\Omega f(u(x,t))\ dx \\
& =l_g-l_f=0\ ,
\end{align*}
which gives, together with the lower bound $u\ge 1$, 
$$u(t) \to 1 \mbox{~~in~~}L^2(\Omega) \mbox{~~as~~}t \to \infty\ .$$
Arguing as above with the help of \eqref{massconservation}, we end up with a contradiction to $\bf (H_1)$.

\medskip

(ii) It is sufficient to establish that $l_g<0$ and $l_f>0$. First we show that $l_g<0$. By the mass conservation property \eqref{massconservation}, we have
$$
l_g =\lim_{t \to \infty} \int_\Omega u(x,t)\ dx - \lim_{t \to \infty} \int_\Omega u^2(x,t)\ dx \le \int_\Omega u_0(x)\ dx\ .
$$
In view of  $\bf (H_3)$, the right-hand side of the above inequality is negative so that $l_g<0$. 

Next we prove that $l_f>0$. Since $u \le 0$ a.e. in $\Omega\times (0,\infty)$, $f(u)=u^2(1-u) \ge 0$ a.e. in $\Omega\times (0,\infty)$ hence $l_f \ge 0$. Assume for contradiction that $l_f=0$. Then 
$$
0 \le \lim_{t\to\infty} \int_\Omega u^2(x,t)\ dx \le \lim_{t \to \infty}\int_\Omega u^2(1-u)(x,t)\ dx =0 \ ,
$$
so that 
$$u(t) \to 0 \mbox{~~in~~}L^2(\Omega) \mbox{~~as~~}t \to \infty.$$
We combine this property with the mass conservation \eqref{massconservation} to conclude that
$$
\int_\Omega u_0(x)\ dx =0\ ,
$$
which contradicts $\bf (H_3)$. Thus $l_f<0$.
\end{proof}

\section{Characterization of $\omega$-limit set}\label{sec4}

We first recall some notation and results from \cite{thanhnam}. Given $u_0$ satisfying $\bf (H_i)$ for some $i=1,2,3$, we denote the corresponding solution to $(P)$ by $u$ and consider the unique solution $Y(t;s)$ of the following auxiliary problem:
\begin{equation}
(ODE) \,\,
\begin{cases}
\dot{Y}(t)= Y(t)^2(1-Y(t))-\lambda(t)Y(t)(1-Y(t)), \quad t > 0,\vspace{6pt}\\
Y(0)=s,
\end{cases}
\end{equation}
where $\dot{Y}:=dY/dt$ and $\lambda$ is defined by \eqref{lambda}. Clearly the function $u$ satisfies
\begin{equation}\label{problem:ODE:22:6:3:14}
u(x,t)=Y(t;u_0(x)) \mbox{~~for a.e.~~} x \in \Omega \mbox{~~and all~~}t \ge 0.
\end{equation}

For later convenience, we introduce the differential operator
$\mathcal L$:
\begin{equation}\label{definition:operator:L}
\mathcal L(Z)(t):=\dot{Z}(t)-Z(t)^2(1-Z(t))+\lambda(t)Z(t)(1-Z(t))\ , \qquad t\ge 0\ .
\end{equation}
The following comparison principle is quite standard in the theory of ordinary differential equations; see, e.g., \cite[Theorem~6.1, page~31]{chap2hale}.

\begin{prop}\label{prop-sosanh}
Let $T > 0$ and let $Z_1, Z_2 \in C^1([0,T])$ satisfy
\begin{equation*}
\begin{cases}
\mathcal L (Z_1)(t) \le \mathcal L (Z_2)(t) \mbox{~~for all~~} t \in [0, T],\vspace{6pt}\\
Z_1(0) \le Z_2(0).
\end{cases}
\end{equation*}
Then
$$Z_1(t) \le Z_2(t) \mbox{~~for all~~} t \in [0, T].$$
\end{prop}

\subsection{$\omega$-limit set when $\bf (H_1)$ holds} \label{subsec1}

We first identify two invariant sets with the help of problem $(ODE)$.

\begin{lemma}\label{lem:for:H1}
Suppose $\bf (H_1)$. We define the sets 
$$\Omega_1(t):=\{x \in \Omega: Y(t;u_0(x))=1\} \quad \mbox{and} \quad \Omega_+(t):=\{x \in \Omega: Y(t;u_0(x))>1\}$$
for each $t\ge 0$. Then 
$$
\Omega_1(t)=\Omega_1(0) \quad \mbox{ and } \quad \Omega_+(t)=\Omega_+(0) \quad\mbox{ for all }\;\; t \ge 0\ .
$$
\end{lemma}
\begin{proof}
Note that if $s=1$, then the unique solution of $(ODE)$ is given by $Y(t;1)=1$ for all $t\ge 0$. Therefore, the conclusion of Lemma~\ref{lem:for:H1} follows from the uniqueness of solutions of $(ODE)$.
\end{proof}

\begin{rem}
Owing to the definition of $Y(\cdot,s)$ we note that
$$
\Omega_1(0) = \{ x\in\Omega\ :\ u_0(x)=1\} \;\;\mbox{ and }\;\; \Omega_+(0) = \{ x\in\Omega\ :\ u_0(x)>1\}\ .
$$
\end{rem}

We next state the main result of this section where we identify $\omega(u_0)$, recalling that we already know that it is a singleton by Theorem~\ref{thm:oml}.

\begin{theorem} \label{thm:forH1}
Suppose $\bf (H_1)$. There holds:
\begin{enumerate}[label=(\alph*)]
\item For all $x \in \Omega_1(0)$, $Y(t;u_0(x)) \to 1$ as $t \to \infty$,
\item For all $x \in \Omega_+(0)$, $Y(t;u_0(x)) \to \lambda_\infty$ as $t \to \infty$.
\end{enumerate}
As a consequence, the only element $\varphi$ in $\omega(u_0)$ is given by 
$$\varphi=\chi_{\Omega_1(0)}+\lambda_\infty \chi_{\Omega_+(0)}\ ,$$
the value of $\lambda_\infty$ being determined by the equation
$$|\Omega_1(0)|+\lambda_\infty |\Omega_+(0)|=\int_\Omega u_0(x)\ dx\ .$$
\end{theorem}

\begin{proof}
The statement (a) is obvious. We prove the statement (b) and first recall that $\lambda_\infty>1$ by Lemma~\ref{lem:lambda:infty}. Let  $x_0 \in \Omega_+(0)$
and let $\e>0$ be arbitrarily small such that $\lambda_\infty-\e>1$.
There exists $t_\e>0$ such that 
$$
\lambda(t) \in [\lambda_\infty-\e, \lambda_\infty+\e] \mbox{~~for all~~} t \ge t_\e.
$$

We define $Y(t):=Y(t;u_0(x_0))$ for all $t \ge t_\e$ and let $\alpha$ and  $\beta$ be the solutions to the ordinary differential equations 
$$\dot{\alpha}=\alpha(1-\alpha)(\alpha-\lambda_\infty+\e)\ , \quad t\ge t_\varepsilon\ , \quad \alpha(t_\e)=Y(t_\e; u_0(x_0))\ ,$$
and
$$\dot{\beta}=\beta(1-\beta)(\beta-\lambda_\infty-\e)\ , \quad t\ge t_\varepsilon\ , \quad \beta(t_\e)=Y(t_\e; u_0(x_0))\ .$$
Since $x_0 \in \Omega_+(0)=\Omega_+(t_\e)$, we have $\alpha(t_\e)=\beta(t_\e)>1$ which implies that 
$$\alpha(t)>1, \quad \beta(t) >1\mbox{~~for all~~} t \ge t_\e,$$
and thus
\begin{equation}\label{eq:pdfr}
\lim_{t \to \infty}\alpha(t)= \lambda_\infty-\e, \quad \quad \lim_{t \to \infty}\beta(t)=\lambda_\infty+\e.
\end{equation}
Let $\mathcal L$ be defined by \eqref{definition:operator:L}; then $\mathcal L(Y)(t)=0$ for all $t \ge 0$. Note that, for $t>t_\varepsilon$, 
$$\mathcal L(\alpha)(t)=-\alpha(t)(1-\alpha(t))(\lambda_\infty-\e-\lambda(t)) \le 0$$
and
$$\mathcal L(\beta)(t)=-\beta(t)(1-\beta(t))(\lambda_\infty+\e-\lambda(t)) \ge 0\ .$$
Since $\alpha(t_\e)=Y(t_\e)=\beta(t_\e)$ it follows from Proposition~\ref{prop-sosanh} that
$$\alpha(t) \le Y(t) \le \beta(t) \mbox{~~for all~~}t \ge t_\e\ .$$
Hence, in view of \eqref{eq:pdfr},
$$\lambda_\infty-2\e \le Y(t) \le \lambda_\infty+2\e \mbox{~~for all~~}t \mbox{~~large enough},$$
which completes the proof of (b). Recalling \eqref{problem:ODE:22:6:3:14} we can then identify the unique element $\varphi$ of $\omega(u_0)$ which is given by 
$$\varphi=\chi_{\Omega_1(0)}+\lambda_\infty \chi_{\Omega_+(0)}.$$
By the mass conservation property \eqref{massconservation}, we have
$$|\Omega_1(0)|+\lambda_\infty |\Omega_+(0)|=\int_\Omega u_0(x)\ dx \ ,$$
which characterizes the value of $\lambda_\infty$.
\end{proof}

\subsection{$\omega$-limit set when $\bf (H_3)$ holds} \label{subsec2}

The following result is similar to Theorem~\ref{thm:forH1}. We omit its proof.
\begin{theorem} 
Suppose $\bf (H_3)$. Then the unique element $\varphi$ of $\omega(u_0)$ is given by 
$$\varphi=\lambda_\infty \chi_{\Omega_-(0)},$$
where $\Omega_-(0):=\{x \in \Omega: u_0(x)<0\}$ and $\lambda_\infty$ is defined by
$$
\lambda_\infty |\Omega_-(0)|=\int_\Omega u_0(x)\ dx.
$$
\end{theorem}

\subsection{$\omega$-limit set when $\bf (H_2)$ holds} \label{subsec3}

We finally turn to the case where $u_0$ satisfies $\bf (H_2)$. As already mentioned, the results to follow are not as precise as in the other cases but still shed some light on the dynamics of $(P)$. 

\begin{lemma} \label{lem:forH2}
Assume that $\bf (H_2)$ holds and that 
$$
l_g := \displaystyle\lim_{t \to \infty} \int_\Omega g(u(x,t))\ dx \neq 0\ .
$$ 
Then 
\begin{enumerate}[label=\emph{(\roman*)}]
\item There is $\lambda_\infty \in [0, 1]$ such that $\lambda(t)\to \lambda_\infty$ as $t \to \infty$.
\item For all $s \in [0, 1]$, there is $Y_\infty(s)\in \{0,\lambda_\infty,1\}$ such that $Y(t;s)\to Y_\infty(s)$ as $t \to \infty$.
\item If $\lambda_\infty\in (0,1)$, then the set $Y_\infty^{-1}(\lambda_\infty)$ contains at most one element.
\end{enumerate}
\end{lemma}

\begin{proof}
(i) Assertion (i) directly follows from the assumption $l_g\ne 0$ and Corollaries~\ref{bounds:lambda(t)} and~\ref{cor:lim:nume:denominateur}.

\bigskip

\noindent (ii) Let $s \in [0, 1]$ and set $Y(t):=Y(t;s)$ for $t\ge 0$. Recall that $(ODE)$ implies that $Y(t)\in (0,1)$ for all $t\ge 0$. Given $\varepsilon\in (0,1)$ there is $t_\varepsilon>0$ such that
\begin{equation}
\lambda_\infty -\varepsilon \le \lambda(t) \le \lambda_\infty+\varepsilon\ , \quad t\ge t_\varepsilon\ .\label{zz1}
\end{equation}

\medskip

\noindent\textbf{Case 1.} Either there are $\varepsilon\in (0,1)$ and $t_0\ge t_\varepsilon$ such that $Y(t_0)\not\in [\lambda_\infty-\varepsilon,\lambda_\infty + \varepsilon]$. 

\medskip

\noindent\textbf{Case~1.1.} If $Y(t_0)<\lambda_\infty-\varepsilon$, then $(ODE)$ and \eqref{zz1} guarantee that $\dot{Y}(t_0)<0$ so that
$$
\tau := \sup\{ t\ge t_0\ :\ Y(t)<\lambda_\infty - \varepsilon\}>t_0\ .
$$
Clearly $\dot{Y}(t)<0$ for $t\in [t_0,\tau)$ from which we readily conclude that $\tau=\infty$. Therefore $Y(t)<\lambda_\infty-\varepsilon$ and $\dot{Y}(t)<0$ for all $t\in [t_0,\infty)$ and we infer from $(ODE)$ and \eqref{zz1} that
$$
\dot{Y}(t) \le Y(t) (Y(t_0)-\lambda_\infty+\varepsilon)\ , \qquad t\ge t_0\ .
$$
Since $Y(t_0)-\lambda_\infty+\varepsilon<0$, we conclude that $Y(t)$ decays exponentially fast to zero and thus that $Y_\infty(s)=0$. 

\medskip

\noindent\textbf{Case~1.2.} If $Y(t_0)>\lambda_\infty+\varepsilon$ then $\dot{Y}(t_0)>0$ by $(ODE)$ and \eqref{zz1} and a similar argument entails that $Y(t)\to 1$ as $t\to 1$. Thus $Y_\infty(s)=1$ in that case.

\medskip

\noindent\textbf{Case~2.} Or $Y(t)\in [\lambda_\infty-\varepsilon,\lambda_\infty+\varepsilon]$ for all $(\varepsilon,t)\in (0,1)\times (t_\varepsilon,\infty)$, which means that $Y(t)\to \lambda_\infty$ as $t\to\infty$ and completes the proof of (ii).

\medskip

\noindent (iii) Assume for contradiction that there are $0<s_1<s_2<1$ such that $Y_\infty(s_1)=Y_\infty(s_2)=\lambda_\infty$. We set $Y_i(t):= Y(t;s_i)$ for $t\ge 0$ and $i=1,2$ and $Z:=Y_2-Y_1$ and notice that  
\begin{equation}\label{eq:Z(t):property}
\lim_{t \to \infty} Z(t)=0, \quad  Z(t) > 0 \mbox{~~for all~~} t \ge 0\ .
\end{equation}
We further deduce from $(ODE)$ that $Z$ solves
$$
\dot{Z} = \left[ (1+\lambda)(Y_1+Y_2) - \lambda - Y_1^2 - Y_1 Y_2 - Y_2^2 \right] Z\ ,
$$
while assertion~(i) of Lemma~\ref{lem:forH2} entails that
$$
\lim_{t\to\infty} \left[ (1+\lambda)(Y_1+Y_2) - \lambda - Y_1^2 - Y_1 Y_2 - Y_2^2 \right](t) = \lambda_\infty (1-\lambda_\infty) > 0\ .
$$
Owing to \eqref{eq:Z(t):property}, we realize that $\dot{Z}(t)\ge 0$ for $t$ large enough, which contradicts \eqref{eq:Z(t):property}.
\end{proof}

\begin{prop}\label{prop:forH2}
Suppose that $\bf (H_2)$ holds and that 
\begin{equation}\label{eq:hypotheses:forH2}
 |\{x \in \Omega: u_0(x)=s \}|=0 \mbox{~~for all~~} s \in (0,1).
\end{equation}
Then 
$$
l_g := \displaystyle\lim_{t \to \infty} \int_\Omega g(u(x,t))\ dx =0 \quad\text{ and }\quad l_f := \displaystyle\lim_{t \to \infty} \int_\Omega f(u(x,t))\ dx =0\ .
$$
\end{prop}

\begin{proof}
Assume for contradiction that
\begin{equation}\label{assump:contradiction}
l_g = \displaystyle\lim_{t \to \infty} \int_\Omega g(u(x,t))\ dx \neq 0\ .
\end{equation}
By Lemma~\ref{lem:forH2}\,(i), there is $\lambda_\infty\in [0,1]$ such that $\lambda(t) \longrightarrow \lambda_\infty$ as $t \to \infty$.

If $\lambda_\infty=0$ or $\lambda_\infty=1$, then it follows from Lemma~\ref{lem:forH2}\,(ii) that $Y_\infty(s)\in \{0,1\}$ for $s \in [0, 1]$. Consequently, $u(1-u)(x,t) = Y(t;u_0(x))(1-Y(t;u_0(x)) \to 0$ as $t\to\infty$ for a.e. $x\in\Omega$ which implies that $l_g=0$ and contradicts \eqref{assump:contradiction}.

If $\lambda_\infty \in (0,1)$ then $Y_\infty^{-1}(\lambda_\infty)$ contains at most one element by Lemma~\ref{lem:forH2}\,(iii) which, together with \eqref{eq:hypotheses:forH2}, implies that $Y_\infty(u_0(x)) \in \{0,1\}$ for a.e. $x \in \Omega$. Consequently $u(1-u)(x,t) = Y(t;u_0(x))(1-Y(t;u_0(x)) \to 0$ as $t\to\infty$ for a.e. $x\in\Omega$ which implies that $l_g=0$ and again contradicts \eqref{assump:contradiction}.

Consequently $l_g=0$ and, since $0\le f(u(x,t))\le g(u(x,t))$ for a.e. $x\in\Omega$ and all $t\ge 0$ by Proposition~\ref{thm:existence:boundedness}\,(ii), we readily obtain that $l_f=0$.
\end{proof}

\begin{rem}
Under the assumptions of Proposition~\ref{prop:forH2} it is unclear whether $\lambda(t)$ has a limit as $t\to\infty$.
\end{rem}

\end{document}